\documentclass[12pt,reqno]{amsart} 
\usepackage{amsmath, amsthm, amscd, amsfonts, amssymb, graphicx, color}
\usepackage{amsmath}
\usepackage{amssymb}
\newtheorem{thm}{Theorem}
\newtheorem{lem}[thm]{Lemma}
\newtheorem{prop}[thm]{Proposition}
\newtheorem{cor}[thm]{Corollary}
\newcommand{\vertiii}[1]{{\left\vert\kern-0.25ex\left\vert\kern-0.25ex\left\vert
#1 \right\vert\kern-0.25ex\right\vert\kern-0.25ex\right\vert}}

\def \diag  {\text {\rm diag}}
\def \e   {\text {\rm e}}

\def \In   {\text {\rm In}}

\def \min   {\text {\rm min}}

\def \diag  {\text {\rm diag}}

\def \tr    {\text {\rm tr}}

\begin{document}

\title[]{Inertia of Kwong matrices}

\author[Rajendra Bhatia]{Rajendra Bhatia}
\address{Ashoka University, Sonepat\\ Haryana, 131029, India}

\email{rajendra.bhatia@ashoka.edu.in}

\author{Tanvi Jain}

\address{Indian Statistical Institute, New Delhi 110016, India}

\email{tanvi@isid.ac.in}



%\date{\today}

\begin{abstract}
Let $r$ be any real number and for any $n$ let $p_1,\ldots,p_n$ be distinct positive numbers.
A Kwong matrix is the $n\times n$ matrix whose $(i,j)$ entry is $(p_i^r+p_j^r)/(p_i+p_j).$
We determine the signatures of eigenvalues of all such matrices.
The corresponding problem for the family of Loewner matrices $\begin{bmatrix}(p_i^r-p_j^r)/(p_i-p_j)\end{bmatrix}$ has been solved earlier.
\end{abstract}
\maketitle

\section{Introduction}

Let $f$ be a nonnegative $C^1$ function on $(0,\infty).$
Let $n$ be a positive integer and $p_1<p_2<\cdots<p_n$ distinct positive real numbers.
The $n\times n$ matrix
\begin{equation}
L_f(p_1,\ldots,p_n)=\begin{bmatrix}\frac{f(p_i)-f(p_j)}{p_i-p_j}\end{bmatrix}\label{eql1}
\end{equation}
is called a {\it Loewner matrix} associated with $f.$
These matrices play an important role in several areas of analysis,
one of them being Loewner's theory of operator monotone functions.
A central theorem in this theory asserts that $f$ is operator monotone if and only if all Loewner matrices associated with $f$ are positive semidefinite.
See \cite{rbh}, \cite{rbh1} and \cite{sim}.

Closely related to Loewner matrices are the matrices
\begin{equation}
K_f(p_1,\ldots,p_n)=\begin{bmatrix}\frac{f(p_i)+f(p_j)}{p_i+p_j}\end{bmatrix}.\label{eqk1}
\end{equation}
These too have been studied in several papers.
In \cite{kwong} Kwong showed that all matrices $K_f$ are positive semidefinite
if (but not only if) $f$ is operator monotone.
Because of this the matrices $K_f$ are sometimes called {\it Kwong matrices}.
Audenaert \cite{a} has characterised all functions $f$ for which all $K_f$ are positive semidefinite.

Of particular interest are the functions $f(t)=t^r,$ where $r$ is any real number.
For these functions we denote $L_f$ and $K_f$ by $L_r$ and $K_r,$ respectively.
Thus
\begin{equation}
L_r(p_1,\ldots,p_n)=\begin{bmatrix}\frac{p_i^r-p_j^r}{p_i-p_j}\end{bmatrix},\label{eql2}
\end{equation}
and\begin{equation}
K_r(p_1,\ldots,p_n)=\begin{bmatrix}\frac{p_i^r+p_j^r}{p_i+p_j}\end{bmatrix}.\label{eqk2}
\end{equation}
Another fundamental theorem of Loewner's says that the function $f(t)=t^r$ is operator monotone 
if and only if $0\le r\le 1.$
Thus all matrices $L_r$ are positive semidefinite if and only if $0\le r\le 1.$
From the work of Kwong and Audenaert cited above it follows that
all matrices $K_r$ are positive semidefinite if and only if $-1\le r\le 1$
and positive definite if and only if $-1<r<1.$

In their work\cite{rbhol} Bhatia and Holbrook studied the matrices $L_r$ for values of $r$ outside the interval $[0,1].$
Among other things, they showed that
when $1<r<2$ every matrix $L_r$ has exactly one positive eigenvalue.
This is in striking contrast to the case $0<r<1,$
in which all eigenvalues of $L_r$ are positive.
This led the authors \cite{rbhol} to make a conjecture about the signature of eigenvalues of $L_r$ as $r$ varies over real numbers.

Let $A$ be any $n\times n$ Hermitian matrix.
The {\it inertia} of $A$ is the triple
$${\rm In} (A) = (\pi(A), \zeta (A), \nu(A)),$$
where $\pi(A),$ $\zeta(A)$ and $\nu(A)$ are respectively the numbers of positive, zero and negative eigenvalues of $A.$
By the results of Loewner cited above
$\In\, L_r=(n,0,0)$ when $0<r<1,$
and the result of Bhatia-Holbrook says that $\In\, L_r=(1,0,n-1)$
when $1<r<2.$
The conjecture in \cite{rbhol} described the inertia of $L_r$ for other values of $r.$
%The question raised in \cite{rbhol} was
%what $L_r$ is for other values of $r.$

In \cite{rbsano} Bhatia and Sano made two essential contributions to this problem.
They provided a better understanding of the problem for the range $1<r<2,$
and they also obtained a solution for the range $2<r<3.$
Let $\mathcal{H}_1$ be the $(n-1)$-dimensional subspace of $\mathbb{C}^n$
defined as
\begin{equation}
\mathcal{H}_1=\left\{x\in\mathbb{C}^n:\sum_{i=1}^{n}x_i=0\right\}.\label{eqr1}
\end{equation}
An $n\times n$ Hermitian matrix is said to be {\it conditionally positive definite} (cpd) if 
$\langle x,Ax\rangle\ge 0$ for all $x\in\mathcal{H}_1.$
It is said to be {\it conditionally negative definite} (cnd) if $-A$ is cpd.
%If $A$ is nonsingular and cpd, then $\In\, A=(n-1,0,1),$
If $A$ is nonsingular and cnd with all entries nonnegative, then $\In\, A=(1,0,n-1).$
Bhatia and Sano \cite{rbsano} showed that the matrix $L_r$ is cnd when $1<r<2,$
thus explaining the result in \cite{rbhol}.
They also showed that $L_r$ is cpd when $2<r<3.$

In the same paper \cite{rbsano} the authors found an interesting difference between the
inertial properties of $L_r$ and $K_r$ in the range $2<r<3.$
They showed that $K_r$ is nonsingular and cnd in the interval $1<r<3,$
and hence $\In\, K_r=(1,0,n-1)$ for such $r.$
Thus, there arises the problem of studying $\In\, K_r$ parallel to that of $\In\, L_r.$

The inertia of $L_r$ was completely determined by Bhatia, Friedland and Jain in \cite {bfj}.
The corresponding theorem on $K_r$ was proved by us shortly afterwards.
This was announced in \cite{bj}.
The aim of the present paper is to publish our proof.
Our main result is the following:

\begin{thm}\label{thmk1}
Let $p_1<p_2<\cdots<p_n$ and $r$ be any positive real numbers and let $K_r$ be the matrix defined in \eqref{eqk2}.
Then 
\begin{itemize}
\item[(i)] $K_r$ is singular if and only if $r$ is an odd integer smaller than $n.$
\item[(ii)] When $r$ is an odd integer smaller than or equal to $n,$
the inertia of $K_r$ is given as follows
$$\In\, K_r=\begin{cases}
 \left(\lceil\frac{r}{2}\rceil, n-r,\lfloor\frac{r}{2}\rfloor\right) & r=1(\, \textrm{mod}\, 4)\\
 \left(\lfloor\frac{r}{2}\rfloor, n-r,\lceil\frac{r}{2}\rceil\right) & r=3\, (\textrm{mod}\, 4).
\end{cases}$$
\item[(iii)] Suppose $k<r<k+2<n,$
where $k$ is an odd integer.
Then
$$\In\, K_r=\begin{cases}
 \left(\lceil\frac{k}{2}\rceil,0,n-\lceil\frac{k}{2}\rceil\right) & k=1\, (\textrm{mod}\, 4)\\
 \left(n-\lceil\frac{k}{2}\rceil,0,\lceil\frac{k}{2}\rceil\right) & k=3\, (\textrm{mod}\, 4).\end{cases}$$
\item[(iv)] If $n$ is odd,
then $\In\, K_r=\In\, K_n$ for $r>n-2$;
and if $n$ is even,
then $\In\, K_r=\In\, K_n=\left(\frac{n}{2},0,\frac{n}{2}\right)$ for $r>n-1.$
%\item[(v)] Every nonzero eigenvalue of $K_r$ is simple.
\end{itemize}
\end{thm}

There is a striking similarity and a striking difference between the behaviour of the signs of eigenvalues of $L_r$ and $K_r.$
As $r$ moves over $(0,\infty)$ the eigenvalues of both flip signs at certain integral values of $r.$
For $L_r$ these flips take place at all integers $r\le n-1,$
and each time all but one eigenvalue change signs.
For $K_r$ the flips take place at all odd integers $r\le n-1.$
At $r=1$ all but one eigenvalue change signs,
and after that all but two eigenvalues change signs.
\vskip0.1in
\begin{center}
\begin{figure}[ht]
\includegraphics[width=4.5in]{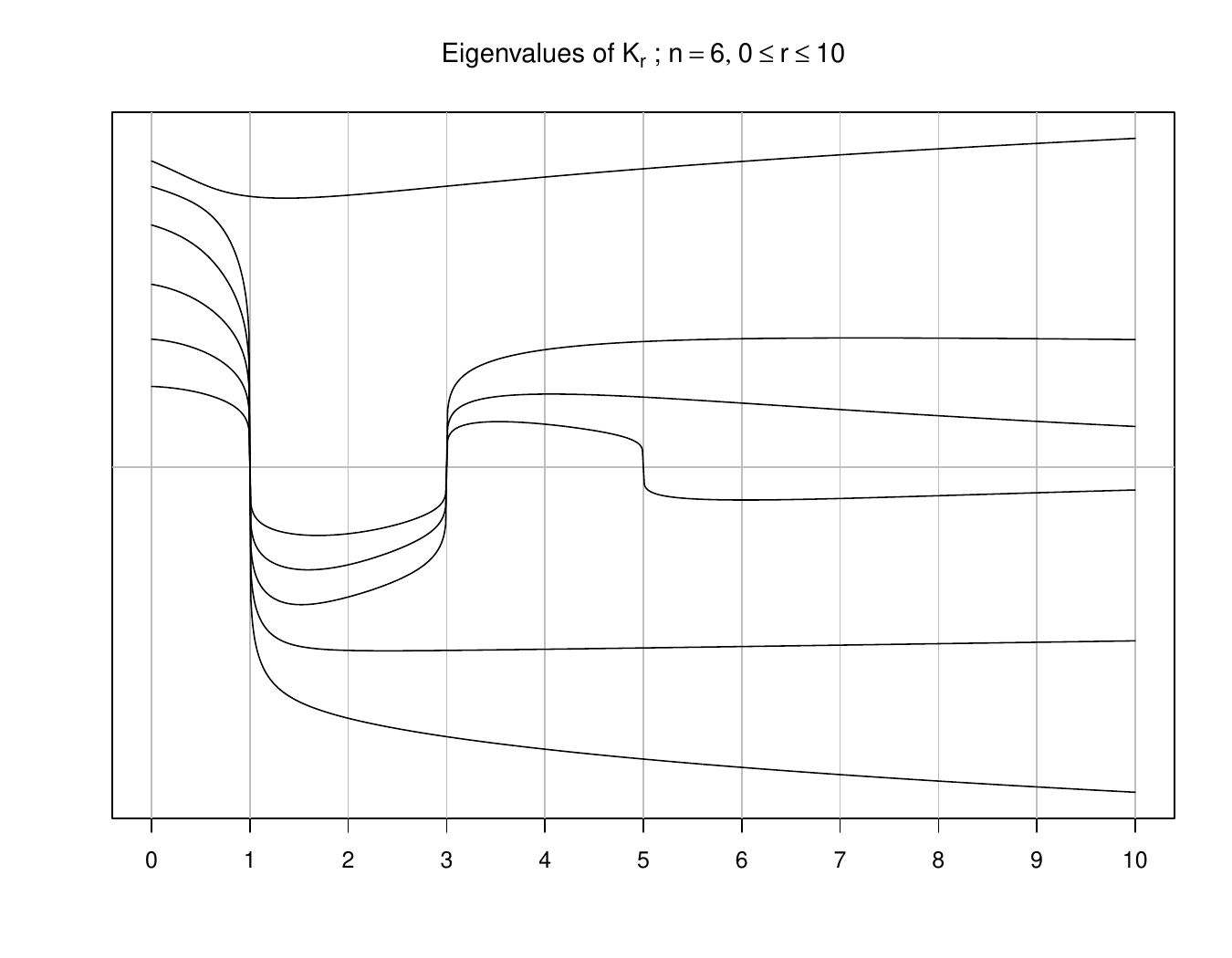}
\caption{}
 \end{figure}
\end{center}

Figure 1 is a schematic representation of the eigenvalues of $K_r$ for $n=6$ and $r\ge 0.$
\vskip.1in
In an earlier paper\cite{bj} we studied the inertia of another family $P_r=\begin{bmatrix}(p_i+p_j)^r\end{bmatrix}.$
The structure of the proof there has been the template of subsequent works on this kind of problem.
See, e.g., \cite{bfj} and the recent work \cite{st} on the Kraus matrix.
(Warning$:$ The authors \cite{st} use the symbol $K_r$ for something different from our Kwong matrix).
Our proof here follows the same steps as in these papers$;$
the details are different at some crucial points.
%Recently, we have seen the work \cite{ks} where a proof of Theorem \ref{thmk1} is offered.
%This is different from our proof at several places.

\section{Proof of Theorem \ref{thmk1}}
Two Hermitian matrices $A$ and $B$ are said to be {\it congruent} if there exists an invertible matrix $X$ such that $B=X^*AX.$
The Sylvester law of inertia says that $A$ and $B$ are congruent if and only if $\In\, A=\In\, B.$
\vskip.1in
Let $D$ be the diagonal matrix $D=\diag(p_1,\ldots,p_n).$ Then for every $r>0$
\begin{equation}
K_{-r}=D^{-r}K_rD^{-r},\label{eqk3}
\end{equation}
and hence, 
\begin{equation}
\In\, K_{-r}=\In\, K_r.\label{eqk4}
\end{equation}
The substitution $p_i=\e^{2x_i},$ $x_i\in\mathbb{R},$ gives
\begin{equation}
K_r=\Delta\tilde{K}_r\Delta,\label{eqk5}
\end{equation}
where $\Delta=\diag\left(\e^{(r-1)x_1},\ldots,\e^{(r-1)x_n}\right),$ and
\begin{equation}
\tilde{K}_r=\begin{bmatrix}\frac{\cosh\, r(x_i-x_j)}{\cosh\, (x_i-x_j)}\end{bmatrix}.\label{eqk6}
\end{equation}
By Sylvester's law $\In\, \tilde{K}_r=\In\, K_r.$
When $n=2$ we have 
$$\tilde{K}_r=\begin{bmatrix}
1 & \frac{\cosh\, r(x_1-x_2)}{\cosh\, (x_1-x_2)}\\
\frac{\cosh\, r(x_1-x_2)}{\cosh\, (x_1-x_2)} & 1\end{bmatrix}.$$
So, $\det\, \tilde{K}_r=1-\frac{\cosh^2r(x_1-x_2)}{\cosh^2(x_1-x_2)}.$
This is positive if $0< r<1,$
zero if $r=1,$
and negative if $r>1.$
The inertia of $K_r$ is $(2,0,0)$ in the first case,
$(1,1,0)$ in the second case, and $(1,0,1)$ in the third.
All assertions of Theorem \ref{thmk1} are thus valid in the case $n=2.$
\vskip.1in
We will use the following extension of Sylvester's law.
A proof is given in \cite{bj}.

\begin{prop}\label{propl1}
Let $n \ge r,$ 
and let $A$ be an $r \times r$ Hermitian matrix and $X$ an $r \times n$ matrix of rank 
$r.$ Then
\begin{equation}
{\rm In} \,\, X^{\ast} AX = {\rm In} \,\, A + (0, n-r, 0).   \label{eq13}
\end{equation}
\end{prop}

We now prove part (ii) of the theorem.
Let $r$ be an odd integer,
$r\le n.$ Then
$$\frac{p_i^r+p_j^r}{p_i+p_j}=p_i^{r-1}-p_i^{r-2}p_j+p_i^{r-3}p_j^2-\cdots + p_j^{r-1}.$$
So the matrix $K_r$ can be factored as
$$K_r=W^*VW,$$
where $W$ is the $r\times n$ Vandermonde matrix given by
$$W=\begin{bmatrix}1 & 1 & \cdots & 1\\
p_1 & p_2 & \cdots & p_n\\
\vdots & \vdots & \vdots\vdots\vdots & \vdots\\
p_1^{r-1} & p_2^{r-1} & \cdots & p_n^{r-1}\end{bmatrix}.$$
and $V$ is the $r\times r$ antidiagonal matrix with entries $(1,-1,1,-1,\cdots,-1,1)$
down its sinister diagonal.
So by the generalised Sylvester's law \eqref{eq13} we have for every odd integer $r\le n,$
$$\In\, K_r=\In\, V+(0,n-r,0).$$
The matrix $V$ is nonsingular and its eigenvalues are $\pm 1.$
In the case $r=1\, (\textrm{mod}\, 4),$
$\tr\, V=1$ and the multiplicity of $1$ as an eigenvalue of $V$ exceeds by one the multiplicity of $-1.$
In the case $r=3\, (\textrm{mod}\, 4),$
$\tr\, V=-1$ and the multiplicity of $-1$ as an eigenvalue of $V$ exceeds by one the multiplicity of $1.$
This establishes part (ii) of Theorem \ref{thmk1}.
\vskip.1in
Next, let $c_1,c_2,\ldots,c_n$ be real numbers,
not all of which are zero,
and let $f$ be the function on $(0,\infty)$ defined as
\begin{equation}
f(x)=\sum_{j=1}^{n}c_j\frac{x^r+p_j^r}{x+p_j}.\label{eqk7}
\end{equation}

\begin{thm}\label{thmk2}
Let $n$ be an odd number.
Then for every positive real number $r>n-1,$
the function $f$ in equation \eqref{eqk7}
has at most $n-1$ zeros in $(0,\infty).$
\end{thm}

\begin{proof}
Consider the function
$g$ defined as
\begin{equation}
g(x)=f(x)\prod_{j=1}^{n}(x+p_j).\label{eqk8}
\end{equation}
Expanding the product,
we can write
\begin{align}
g(x)&=\alpha_0+\alpha_1x+\cdots+\alpha_{n-1}x^{n-1}\nonumber\\
&\ +\beta_0x^r+\beta_1x^{r+1}+\cdots+\beta_{n-1}x^{r+n-1}.\label{eqk9}
\end{align}
The function $g$ can be written as
$$g(x)=x^rh_1(x)+h_2(x),$$
where
$$h_1(x)=\sum_{i=1}^{n}c_i\prod_{j\ne i}(x+p_j)\textrm{  and  }h_2(x)=\sum_{i=1}^{n}c_ip_i^r\prod_{j\ne i}(x+p_j).$$
Since both $h_1$ and $h_2$ are Lagrange interpolation polynomials of degree at most $n-1$ and not all $c_i$ are zero,
neither of the polynomials $h_1$ and $h_2$ is identically zero.
Hence $g$ is not identically zero.
%We can argue as in the proof of Theorem 2.2 of \cite{bfj} that not all coefficients in \eqref{eqk9}
%are zero.
%If $f$ and $g$ are the functions defined in eqref{eq14} and eqref{eq16},
%respectively,
Now, consider the function $g_0$ defined as
$$g_0(x)=\sum_{i=1}^{n}c_i\frac{x^r-p_i^r}{x-p_i}\prod_{j=1}^{n}(x-p_j).$$
Then a calculation shows that
\begin{align}
g_0(x)&=-\alpha_0+\alpha_1x+\cdots+\alpha_{n-2}x^{n-1}-\alpha_{n-1}x^{n-1}\nonumber\\
&\ +\beta_0x^r-\beta_1x^{r+1}+\cdots+\beta_{n-1}x^{r+n-1}.\label{eqk10}
\end{align}
By the Descartes rule of signs (\cite{gg}, p.46), the number of positive zeros of $g$ is no more than the number of sign changes in the sequence of coefficients
$$\left(\alpha_0,\alpha_1,\ldots,\alpha_{n-1},\beta_0,\beta_1,\ldots,\beta_{n-1}\right).$$
Let this number of sign changes be $s,$
and let $s_0$be the number of sign changes in the coefficients in \eqref{eqk10}.
Since $n$ is odd, we have 
$s+s_0\le 2n-1.$
We know that $g_0$ has at least $n$ positive zeros $p_1,\ldots,p_n.$
So, $s_0\ge n,$
and hence $s\le n-1.$
Hence $g$ has at most $n-1$ positive zeros,
and therefore so does $f.$
\end{proof}

We can deduce the following.

\begin{cor}\label{cork3}
Let $n$ be an odd number, and let $p_1,\ldots,p_n$ and $q_1,\ldots,q_n$ be two $n$-tuples of distinct positive numbers.
Then for every $r>n-1,$
the $n\times n$ matrix
\begin{equation}
\begin{bmatrix}\frac{p_i^r+q_j^r}{p_i+q_j}\end{bmatrix}\label{eqn1}
\end{equation}
is nonsingular.
So, in particular if $n$ is odd,
then for every $r>n-1,$ the matrix $K_r$ is nonsingular.
\end{cor}

\begin{proof}
If the matrix \eqref{eqn1} is singular, then there exists a nonzero tuple $(c_1,\ldots,c_n)$ such that
$$f(x)=\sum_{j=1}^{n}c_j\frac{x^r+q_j^r}{x+q_j}$$
has at least $n$ zeros $p_1,\ldots,p_n.$
But this is not possible by Theorem \ref{thmk2}.
So, the matrix \eqref{eqn1} and hence, the matrix $K_r$ is nonsingular for all odd $n$ and $r>n-1.$
\end{proof}

%The next ingredient needed for our proof is the identity
We complete the proof of Theorem \ref{thmk1}
using a "snaking" process$:$
the validity of the theorem is extended by alternatively increasing $n$ and $r.$
\vskip.1in
For any positive numbers $p$ and $q$ and any real $r,$ we have
$$\frac{p^r+q^r}{p+q}=p^{r-1}-p\frac{p^{r-2}+q^{r-2}}{p+q}q+q^{r-1}.$$
This gives us the identity
\begin{equation}
K_r=D^{r-1}E-DK_{r-2}D+ED^{r-1},\label{eqk11}
\end{equation}
where $D$ is the diagonal matrix $\diag(p_1,\ldots,p_n)$
and $E$ the matrix all whose entries are one.
For $1\le j\le n,$
let $\mathcal{H}_j$ be the subspace of $\mathbb{C}^n$ defined as
\begin{align*}
 \mathcal{H}_j &= \left \{ x : \sum x_i = 0, \sum p_i x_i = 0, \ldots,\sum p_i^{j-1} x_i 
= 0   \right \} \\
&=\left \{ x : Ex = 0, E Dx = 0, \ldots,E D^{j-1} x = 0    \right \}.
\end{align*}
Evidently, $\textrm{dim}\, \mathcal{H}_j=n-j$ and $\mathcal{H}_{j+1}\subset\mathcal{H}_j.$
\vskip.1in
It will be convenient to use the notation$K_r^{(n)}$ to indicate an $n\times n$ matrix of the type $K_r.$
When the superscript $n$ is not used,
it will be understood that a statement about $K_r$ is true for all $n.$
\vskip.1in
Recall that $K_r$ is known to be positive definite for $0<r<1.$
The relation \eqref{eqk3} then shows that it has the same property for $-1<r< 0.$
$K_0$ is the Cauchy matrix $\begin{bmatrix}\frac{2}{p_i+p_j}\end{bmatrix}$ and is positive definite.
Thus $K_r$ is positive definite for $-1<r<1.$
\vskip.1in
Now let $1<r<3.$
Then $-1<r-2<1.$
Using the identity \eqref{eqk11} we see that
if $x$ is a nonzero vector in $\mathcal{H}_1,$
then
$$\langle x,K_rx\rangle=-\langle Dx,K_{r-2}Dx\rangle<0.$$
So, the matrix $K_r$ is conditionally negative definite and has at least $n-1$ negative eigenvalues.
Since all entries of $K_r$ are positive, it has at least one positive eigenvalue.
Thus
\begin{equation}
\In\, K_r=(1,0,n-1),\label{eqk12}
\end{equation}
for $1<r<3.$
By Corollary \ref{cork3},
$K_r^{(3)}$ is nonsingular for $r>2.$
So, $\In\, K_r^{(3)}$ does not change for $r>2.$
This shows that $\In\, K_r^{(3)}=(1,0,2)$ for all $r>1.$
So, the theorem is established when $n=3.$
\vskip.1in
Now let $n>3$ and $3<r<5.$
Using the identity \eqref{eqk11} and the case $1<r<3$ of the theorem that has been established
we see that if $x$ is a nonzero vector in $\mathcal{H}_2,$
then $\langle x, K_rx\rangle>0.$
So, the matrix $K_r$ has at least $n-2$ positive eigenvalues.
By the $n=3$ case already proved,
we know that $K_r$ has a $3\times 3$ principal submatrix
with two negative eigenvalues.
So, by Cauchy's interlacing principle,
$K_r$ must have at least two negative eigenvalues.
We conclude that
\begin{equation}
\In\, K_r=(n-2,0,2),\textrm{ if }3<r<5.\label{eqk14}
\end{equation}
In particular, this shows that $\In\, K_r^{(5)}=(3,0,2)$ if $3<r<5,$
and since $K_r^{(5)}$ is nonsingular for $r>4,$
it has the same inertia for all $r>3.$
So, Theorem \ref{thmk1} is established when $n=5.$
Next consider the case $n=4.$
Let $r>3.$
The matrix $K_r^{(4)}$ has a principal submatrix $K_r^{(3)}$ whose inertia is $(1,0,2).$
So, by the interlacing principle $K_r^{(4)}$ has at least two negative eigenvalues.
On the other hand $K_r^{(4)}$ is a principal submatrix of $K_r^{(5)}$
whose inertia is $(3,0,2).$
So, again by the interlacing principle $K_r^{(4)}$ has at least two positive eigenvalues.
Thus $\In\, K_r^{(4)}=(2,0,2)$ for all $r>3,$
and Theorem \ref{thmk1} is established for $n=4.$
\vskip.1in
This line of reasoning can be continued.
Use the space $\mathcal{H}_3$ at the next stage to go to the interval $5<r<7.$
Then use the established case $n\le 5$ to extend the validity of the theorem to first the case $n=7,$ and then $n=6.$

\section{Remarks}

\begin{enumerate}

\item[1.] In Theorem 1.1(v)  of \cite {bfj} it was shown that all nonzero eigenvalues of $L_r$ are simple.
We have not been able to prove a corresponding statement for $K_r.$
\vskip.1in
\item[2.] An $n \times n$ real matrix $A$ is said to be {\it strictly sign-regular} (SSR) if for every $1\le k\le n,$ all $k \times k$ sub-determinants of 
$A$ are nonzero and have the same sign. If this is true for all $1\le k\le m$ for some $m< n,$ then we say $A$ is in the class $\textrm{SSR}_{\mbox{m}}.$See \cite{k} for a detailed study of such matrices.
In \cite{bfj} it was shown that the matrix $L_r$ is in the class $\textrm{SSR}_{\mbox{r}}$ if $r=1,2,\ldots,n-1,$
and in the class SSR for all other $r>0.$
This fact was then used to prove the simplicity of nonzero eigenvalues of $L_r.$

Let $n=4$ and consider the matrix $K_3(1,2,5,10).$
It can be seen that the leading $2\times 2$ principal subdeterminant of this matrix is $-5,$ while the determinant of the top right $2\times 2$ submatrix is $35.$
So this matrix is not in the class $\textrm{SSR}_{\mbox{2}}.$
\vskip.1in
\item[3.] Let $p_1<p_2$ and $q_1<q_2$ be two ordered pairs of distinct positive numbers such that $\{p_1,p_2\}\cap\{q_1,q_2\}$ is nonempty.
With a little work it can be shown that the determinant of the $2\times 2$ matrix $\begin{bmatrix}\frac{p_i^r+q_j^r}{p_i+q_j}\end{bmatrix}$ is positive if $0<r<1$ and negative if $r>1.$
Using this one sees that for $n=3$ and $r\ne 1,$ the matrix $K_r$ is SSR.
\vskip.1in
\item[4.] There is a curious and intriguing connection between the inertia of $K_r$ and that of another family.
For $r\ge 0,$
let $B_r$ be the $n\times n$ matrix
\begin{equation}
B_r=\begin{bmatrix}|p_i-p_j|^r\end{bmatrix}.\label{eqb1}
\end{equation}
This family has been studied widely in connection with interpolation of scattered data and splines.
The inertias of these matrices were studied by Dyn, Goodman and Micchelli in \cite{dyn}.
In view of our theorem their result can be stated as
\begin{equation}
\In\, B_r=\In\, K_{r+1}\label{eqbk}
\end{equation}
for all $r\ge 0.$
It will be good to have an understanding of what leads to this remarkable equality.
By Sylvester's law \eqref{eqbk} is equivalent to saying that $B_r$ and $K_{r+1}$ are congruent.
In a recent work \cite{ks} the authors construct an explicit congruence between these two matrices, thus providing an alternative proof of \eqref{eqbk}.
%An explicit congruence has been worked out in the paper \cite{ks}.

\end{enumerate}


\begin{thebibliography}{}

\bibitem{a}
K. M. R. Audenaert, {\it A characterisation of anti-Loewner functions}, Proc. Amer. Math. Soc., 139 (2011) 4217-4223.

\bibitem{rbh} 
R. Bhatia, {\it Matrix Analysis,} Springer, 1997.

\bibitem{rbh1} 
R. Bhatia, {\it Positive Definite Matrices,} Princeton
University Press, 2007.

\bibitem{bfj}
R. Bhatia, S. Friedland and T. Jain, {\it Inertia of Loewner matrices}, Indiana Univ. Math. J., 65 (2016) 1251-1261. 

\bibitem{rbhol} 
R. Bhatia and J. A. Holbrook, {\it Fr\'echet derivatives of
the power function,} Indiana Univ. Math. J., 49(2000) 1155-1173.

\bibitem{bj} R. Bhatia and T. Jain, {\it Inertia of the matrix $\left [
(p_i+p_j)^r\right ],$} J. Spectr. Theory, 5 (2015) 71-87.

\bibitem{rbsano} R. Bhatia and T. Sano, {\it Loewner matrices and operator
convexity,} Math. Ann., 344 (2009) 703-716.

%\bibitem{rbsano1}
%R. Bhatia and T. Sano, {\it Positivity and conditional positivity of loewner matrices}, Positivity, 14 (2010) 421-430.

\bibitem{dyn} N. Dyn, T. Goodman and C. A. Micchelli,  {\it Positive powers of
certain conditionally negative definite matrices,} Indag. Math., 48 (1986)
163-178.


\bibitem{k}
S. Karlin, {\it Total Positivity,} Stanford University Press, 1968.
\bibitem{kwong} M. K. Kwong, {\it Some results on matrix monotone functions,}
Linear Algebra Appl., 118 (1989) 129-153.

\bibitem{ks}
Mandeep, Y. Kapil and M. Singh, {\it On a question of Bhatia and Jain III}, preprint.

\bibitem{gg} G. P\'olya and G. Szeg\"o, {\it Problems and Theorems in Analysis,}
Volume II, 4th ed., Springer, 1976.

\bibitem{st}
T. Sano and K. Takeuchi, {\it Inertia of Kraus matrices}, J. Spectr. Theory, 12 (2022) 1443-1457.

\bibitem{sim}
B. Simon, {\it Loewner's Theorem on Monotone Matrix Functions}, Springer, 2019.
\end{thebibliography}
\end{document}